\documentclass{proc-l}
\usepackage{amssymb,verbatim}
\usepackage{graphicx}

\newtheorem{thm}{Theorem}[section]
\newtheorem{lem}[thm]{Lemma}
\newtheorem{cor}[thm]{Corollary}

\theoremstyle{definition}
\newtheorem{dfn}[thm]{Definition}
\newtheorem{example}[thm]{Example}

\newcommand{\uacs}{uniquely arcwise connected space}

\newcommand{\R}{\mathbb{R}}
\newcommand{\Z}{\mathbb{Z}}

\newcommand{\End}{\mathrm{End}}

\newcommand{\ar}{\mathrm{Ar}}

\newcommand{\orb}{\mathrm{O}}

\newcommand{\ord}{\mathrm{ord}}

\newcommand{\U}{\mathcal{GD}}

\newcommand{\go}{\mathrm{GO}}

\newcommand{\ol}{\overline}

\newcommand{\0}{\varnothing}
\newcommand{\sm}{\setminus}

\newcommand{\bd}{\mathrm{Bd}}

\newcommand{\ch}{\mathrm{Ch}}

\newcommand{\om}{\omega}

\newcommand{\C}{\mbox{$\mathbb{C}$}}

\newcommand{\uc}{\mathbb{S}^1}

\def\Cc{\mathcal{C}}
\def\Rc{\mathcal{R}}

\begin{document}

\date{August 26, 2013; revised May 23, 2014}
\title[Pointwise-recurrent maps of generalized dendrites]
{Pointwise-recurrent maps on uniquely arcwise connected locally
arcwise connected spaces}

\author[A.~M.~Blokh]{Alexander~M.~Blokh}

\thanks{The author was partially
supported by NSF grant DMS--1201450}

\address[Alexander~M.~Blokh]
{Department of Mathematics\\ University of Alabama at Birmingham\\
Birmingham, AL 35294-1170}

\email[Alexander~M.~Blokh]{ablokh@math.uab.edu}




\begin{abstract}
We prove that self-mappings of uniquely arcwise connected locally
arcwise connected spaces are pointwise-recurrent if and only if all
their cutpoints are periodic while all endpoints are either periodic
or belong to what we call ``topological weak adding machines''. We
also introduce the notion of a \emph{ray complete} uniquely arcwise
connected locally arcwise connected space and show that for them the
above ``topological weak adding machines'' coincide with classical
adding machines (e.g., this holds if the entire space is compact).
\end{abstract}

\maketitle

\section{Introduction and the main results}\label{s:intro}

There are two main types of results in interval dynamics. First,
these are facts about periodic points (for a map $f$, a point $x$ is
called \emph{($f$-)periodic (of period $n>0$)} if $f^n(x)=x$ and
$f^t(x)\ne x$ for all $0<t<n$). The first step here was the
celebrated \emph{Sharkovsky Theorem} \cite{sha64} on the coexistence
among periods of periodic points of an interval map. The Sharkovsky
Theorem started \emph{combinatorial one-dimensional dynamics} (see a
nice book \cite{alm00} with an extensive list of references). One
direction in which the field has developed is the study of the
coexistence among periods of periodic points for self-mappings of
``graphs'', i.e. \emph{one-dimensional compact branched manifolds}.

Results of the second type deal with all limit sets rather than only
periodic orbits. This direction has also been initiated by
Sharkovsky, who studied maps of the interval from this perspective
in a number of papers (see, e.g., \cite{sha64a, sha66, sha66a,
sha67, sha68}); the scope of our work does not allow us to go into a
detailed description of this series of articles which, in our view,
laid the foundation of the \emph{one-dimensional topological
dynamics}.

It is natural to see if bounds of one-dimensional topological
dynamics can be pushed further to cover other (one-dimensional)
spaces. As was mentioned, in some works one-di\-men\-sio\-nal
to\-po\-lo\-gi\-cal dy\-na\-mics is stu\-died for ``graphs''(see,
e.g., \cite{alm00, blo80xx}). In this paper we consider a specific
dynamical problem for one-dimensional spaces which can be viewed as
more complicated than ``graphs''.

All topological spaces considered in this paper are assumed to be
Hausdorff. By an \emph{arc} we mean a homeomorphic image of $[0,
1]$; by a \emph{Peano subset} we mean a continuous image of $[0,
1]$. A very good reference here is Chapter 3 of \cite{hy88}.


\begin{dfn}[Uniquely arcwise connected spaces]\label{d:uac}
If for any points $x, y\in X$ there exists an arc $I\subset X$ with
endpoints $x, y$, then $X$ is called \emph{arcwise connected}; if
$I$ is \emph{unique}, then $X$ is called \emph{uniquely arcwise
connected}.
\end{dfn}

To give examples we need the following definition.

\begin{dfn}[Endpoints, cutpoints, branchpoints]\label{d:dendr}
A point $x\in X$ is said to be of \emph{order $\ord_X(x)$ in $X$} if
there are $\ord_X(x)$ components of $X\sm \{x\}$. A point $x\in X$
is called an \emph{endpoint} of $X$ if $\ord_X(x)=1$, a
\emph{cutpoint} of $X$ if $\ord_X(x)>1$, and a \emph{branchpoint} of
$X$ if $\ord_X(x)>2$. 
\end{dfn}

Dendrites and trees are known uniquely arcwise connected spaces.

\begin{dfn}[Dendrites and trees]\label{d:uacex}
A \emph{dendrite} is a non-de\-ge\-ne\-ra\-ted lo\-cal\-ly
con\-nected continuum containing no Jordan curves. A \emph{tree} is
a dendrite with finitely many branchpoints.
\end{dfn}

A lot of arcwise-connected spaces are neither trees nor dendrites.

\begin{dfn}[Locally arcwise connected spaces]\label{d:rtree}
A topological space $X$ is \emph{locally arcwise connected} if any
point has a basis of arcwise connected open sets.
\end{dfn}


We study uniquely arcwise connected locally arcwise connected
topological Hausdorff spaces. At the suggestion of J. Mayer and L.
Oversteegen we call such spaces \emph{generalized dendrites} and
denote the family of all such spaces by $\U$ (it is easy to see that
dendrites belong to $\U$). We rely upon various properties of
uniquely arcwise connected spaces and generalized dendrites which we
now list together with useful notation. Despite their sometimes
complicated structure, uniquely arcwise con\-nec\-ted spa\-ces
allow, by their nature, for nice notation of their subarcs.

\begin{dfn}[Arcs and notation for them]\label{d:subarc}
Let $X$ be uniquely arcwise connected. Then for any points $a\ne
b\in X$ a unique closed arc in $X$ with endpoints $a$ and $b$ is
denoted $[a, b]$; the notation $(a, b), (a, b]$ and $[a, b)$ is
analogous to similar notation in the interval case. Moreover, a
homeomorphism $\alpha:[0, 1]\to [a, b]$ with $\alpha(0)=a,
\alpha(1)=b$ induces the order on $[a, b]$. 
\end{dfn}

If $K\subset X$ is a Peano subset and $a, b\in K$ then $[a,
b]\subset K$ (e.g., if $c, d\in [a, b]$ then $[c, d]\subset [a,
b]$). Thus, any Peano subset of $X$ is uniquely arcwise connected.

\begin{dfn}[Path component]\label{d:arc}
Let $X$ be a topological space. Then a maximal by inclusion arcwise
connected subset of $X$ is called an \emph{path component} of $X$.
Thus, for a point $x\in X$, the path component of $X$ containing $x$
is the union of all arcs in $X$ containing $x$. Also, if $x, a, b\in
K\subset X$ then we say that \emph{$x$ separates $a$ and $b$ in $K$}
if $a$ and $b$ belong to distinct path components of $K\sm \{x\}$.
\end{dfn}

If $X$ is uniquely arcwise connected and $a, b\in X$, then $u$
separates $a$ and $b$ in $X$ if and only if $u\in (a, b)$. Hence
$[a, x]\cup [x, b]=[a, b]$ if and only if $x\in [a, b]$, and if
$y\ne x$ then either $x$ separates $a$ and $y$, or $x$ separates $b$
and $y$. Later in the paper we will use the following simple fact:
if $a, b, c, d\in X$ with $b\in (a, c)$ and $c\in (b, d)$ then $b,
c\in [a, d]$. 

\begin{dfn}[Arc cutpoints]\label{d:arccutpts}
For a point $x\in X$, denote by $\ord_{X, arc}(x)$ the number of
path components of $X\sm \{x\}$. Depending on $\ord_{X, arc}(x)$, we
call $x$ an \emph{arc endpoint of $X$} (if $\ord_{X, arc}(x)=1$), an
\emph{arc cutpoint of $X$} (if $\ord_{X, arc}(x)>1$), and an
\emph{arc branchpoint of $X$} (if $\ord_{X, arc}(x)>2$).
\end{dfn}

Clearly, $x$ is an arc cutpoint if and only if there exist $a, b$
such that $x\in (a, b)$. Cutpoints are arc cutpoints, but the
opposite is not always true.

\begin{example}[``Warsaw Circle'']\label{e:warsaw}
The ``Warsaw Circle'' $W_c$ is defined as follows: take the graph of
the function $\sin(\frac{1}{x}), 0<x\le \frac{1}{2\pi}$, add to it a
vertical segment from $(0, -1)$ to $(0, 1)$, and then complete the
thus constructed continuum $C$ with an arc $I$ connecting $(0, -1)$
with $(\frac{1}{2\pi}, 0)$ and avoiding $C$. Then no point of $W_c$
is a cutpoint of $W_c$, but any point of $W_c\sm \{(0, 1)\}$ is its
arc cutpoint. Observe that $W_c$ is uniquely arcwise connected, but
not locally arcwise connected. For generalized dendrites the
situation is different.
\end{example}


\begin{lem}\label{l:cutarcut} If $X$ is a generalized dendrite
then all components of an open set $U\subset X$ are open and are
generalized dendrites; thus, path components of $X\sm \{x\}$ are
components of $X\sm\{x\}$, are open and locally arcwise connected
(so that $\ord_{X, arc}(x)=\ord_{X}(x)$). In particular, the sets of
arc endpoints, arc cutpoints, and arc branchpoints of $X$ coincide
with the sets of endpoints, cutpoints and branchpoints of $X$,
respectively. Moreover, if $A$ is a component of $X\sm \{x\}$ then
$\ol{A}=A\cup \{x\}$.
\end{lem}


\begin{proof}
We claim that, for $x\in X$ and $A$, a path component $A$ of $X\sm
\{x\}$ is open. Take a point $a\in A$. Since $X$ is locally arcwise
connected, we can find a neighborhood $U$ of $a$ in $X$ such that
$U$ is arcwise connected and $x\notin U$. This implies that
$U\subset A$ and shows that $A$ is open and locally arcwise
connected. The set $X\sm \{x\}$ decomposes into pairwise disjoint
path components of $X\sm \{x\}$ each of which is open, connected and
locally arcwise connected. Hence path components of $X\sm \{x\}$ are
components of $X\sm \{x\}$ with desired properties. Let $U\subset X$
be open and let $C\subset U$ be a component of $U$. Every point
$x\in C$ comes into $C$ with a small arcwise connected neighborhood.
Thus, $C$ is locally arcwise connected and open. If $x, y\in C$ but
there exists $z\in [x, y]\sm C$ then $x, y$ belong to distinct path
components of $X\sm \{z\}$ and $C$ is not connected, a
contradiction. Finally, if $A$ is a component of $X\sm \{x\}$ then
the complement of $A\cup \{x\}$ is open as the union of all other
path components of $X\sm \{x\}$. Hence $A\cup \{x\}$ is closed;
since $x\in \ol{A}$ we get the desired.
\end{proof}

Few dynamical results were obtained for continuous maps on dendrites
(see, e.g. \cite{mt89, aeo07}). So-called \emph{$\R$-trees} give
another example of generalized dendrites; however results on
$\R$-trees are either not dynamical (see, e.g., \cite{nik89},
\cite{mo90} or \cite{mno92}) or arise in the study of groups of
isometries of hyperbolic space \cite{thu88, ms84, ms88, bes88} and
do not deal with the dynamics on $\R$-trees. The author is not aware
of any dynamical results for generalized dendrites. However their
one-dimensional nature allows one to consider for them some
classical problems of topological dynamics. To describe a particular
problem which we tackle in this paper, we need more definitions.
Recall that a point $z$ is called a  \emph{limit} point of a
sequence $x_0, x_1, \dots$ if in any neighborhood of $z$ there
exists a point $x_i\ne z$.

\begin{dfn}[Recurrent points; pointwise-recurrent maps]\label{d:limrec}
Let $f:X\to X$. Given a point $x\in X$, the sequence $(x, f(x),
\dots)=\orb_f(x)=\orb(x)$ is called the \emph{($f$-)orbit} of $x$.
The set $\om_f(x)=\om(x)$ of all limit points of $\orb(x)$ is said
to be the \emph{($\om$-)limit set of $x$}. A point which belongs to
its own limit set is said to be \emph{recurrent} (in other words, a
point $x$ such that $f(x)$ visits any neighborhood of $x$ is said to
be \emph{recurrent}). A map such that all points are recurrent is
called \emph{pointwise-recurrent}.
\end{dfn}

An important and nice property of recurrent points is the following
theorem due to Gottschalk, Erd\"os and Stone.

\begin{thm}[\cite{es45, got44}]\label{t:recpow} If $g$ is a continuous map of a
Hausdorff topological space then, for any positive integer $n$, the
set of recurrent points of $g$ and the set of recurrent points of
$g^n$ coincide.
\end{thm}

The most obvious example of a recurrent point is a periodic point;
in this case the recurrence manifests itself in the most transparent
way. Accordingly, an easy example of a pointwise-recurrent map is a
one-to-one map of a finite set as in this case all points are
periodic. A more complicated case is that of a \emph{minimal} map,
i.e. such a map $g:X\to X$ that all points of $X$ have dense orbit
in $X$. This shows that in general pointwise-recurrent maps can have
a complicated nature.

However with some additional restrictions on the space (often
assumed a \emph{manifold} or a \emph{continuum}) and the map (often
assumed a \emph{homeomorphism}) one can establish a close connection
between pointwise-recurrent maps and maps whose all points (or vast
majority of points) are periodic. In some cases it is even possible
to show that their periods are uniformly bounded; a lot of classic
results are obtained for pointwise-recurrent homeomorphisms along
these lines (see, e.g., \cite{kp98, mon37, ot90, wea72}). The aim of
this paper is to show that if we replace the restriction on the map
(normally required to be a homeomorphism) by that on the space
(required by us to be from $\U$) we can still obtain similar
results. This reconfirms a heuristic observation according to which
in a lot of cases results valid for homeomorphisms of higher
dimensional spaces have analogs for continuous maps of
one-dimensional spaces.

Given a map $f:X\to X$, define the \emph{grand orbit
$\go_f(x)=\go(x)$} as the set of all points which eventually map to
$\orb(x)$. A set $A\subset X$ is \emph{invariant} if $x\in A$
implies $\orb(x)\subset A$ (equivalently, $f(A)\subset A$). A set
$B\subset X$ is \emph{fully invariant} if $y\in B$ implies
$\go(y)\subset X$ (equivalently, $f(B)\cup f^{-1}(B)\subset B$). Our
arguments will be based, in particular, on the fact that
point\-wise-recurrent maps have some restrictive properties which
can be used in their description. Indeed, suppose that $f:X\to X$ is
pointwise-recurrent. Let us show that then for any point $x\in X$ we
have $\go(x)\subset \om(x)$. Indeed, let $y\in \go(x)$. Then
$\om(y)=\om(x)$ while, on the other hand, the fact that $y$ is
recurrent implies that $y\in \om(y)$. Hence $y\in \om(x)$. Also, it
follows that $f(X)$ is dense as otherwise a point from $X\sm
\ol{f(X)}$ is not recurrent. In particular, if $f(X)$ is closed then
$f(X)=X$. This yields the following property.

\smallskip

\noindent\textbf{Property A}. \emph{Let $f:X\to X$ be a
pointwise-recurrent self-mapping of $X$. 
Then $\ol{f(X)}=X$ and $\go(x)\subset \om(x)$ for any $x\in X$ so
that any periodic orbit is fully invariant under $f$. In particular,
if it is known that $f(X)$ is closed then $f(X)=X$.}

\smallskip

Property A can be used to characterize pointwise-recurrent
self-mappings of $X$.

\begin{lem}\label{l:interval}
A continuous map $f:[0, 1]\to [0, 1]$ is pointwise-recurrent if and
only if $f^2$ is the identity map.
\end{lem}

\begin{proof}
Suppose that $f$ is not the identity map. We may assume that one of
the following two cases holds: (1) there exists an interval $(a, b)$
which is fixed point free and such that at least one of the points
$a, b$ is fixed and $f(a, b)\cap (a, b)\ne \0$, or (2) such interval
does not exist, the set of all fixed points is a point $d$, and the
map $f$ ``flips'' $[0, d]$ and $[d, 1]$.

Indeed, $[0, 1]=Y\cup Z$ where $Y\ne \0$ is the set of all fixed
points of $f$ and $Z\ne 0$ is an at most countable union of open
intervals whose endpoints are fixed except possibly for the
intervals with endpoints $0$ or $1$. If $Y$ is not connected, we can
find a component $(a,b)$ of $Z$ whose both endpoints are fixed, and
(1) holds. Otherwise suppose that $Y=[u, v]$. We may assume that
$0<u$. If $f(0, u)\cap (0, u)\ne \0$ then again (1) holds. Otherwise
$f[0, u)\subset [u, 1]$. If $u<v$ this implies by continuity that
there are points in $(0, u)$ which are not fixed but map to fixed
points in $[u, v]$, a contradiction with Property A. Hence in this
case $u=v$ and $f(0, u)\subset (u, 1]$. Therefore $u<1$ and similar
arguments show that $f(u, 1]\subset [0, u)$, i.e. in the end case
(2) holds. Consider these cases separately.

(1) Let $a$ be fixed; assume that $(a, b)$ is maximal by inclusion
interval with listed properties. By Property A and since $f(a,
b)\cap (a, b)\ne \0$ we see that $f((a, b))>a$. Since $(a, b)$ is
fixed point free, all its points map in the same direction by $f$.
If they map towards $a$ then they are attracted by $a$ and clearly
there are non-recurrent points. If $f(x)>x$ for any $x\in (a, b)$
then, since $(a, b)$ is maximal, either $b<1$ is fixed or $b=1$
which also forces $b$ to be fixed. Similar to the above this implies
that, by Property A, $f((a, b))<b$ and hence all points of $(a, b)$
are attracted by $b$ and there are non-recurrent points.

(2) By Property A we see that $f^2([0, d])\subset [0, d]$ and
$f^2([d, 1])\subset [d, 1]$. Now (1) leads to a contradiction unless
$f^2$ is the identity map.
\end{proof}

I. Naghmouchi \cite{nag12} recently obtained far more general
results. Namely, let $D$ be a dendrite whose set of endpoints
$\End(D)$ is countable; the first result of \cite{nag12} is that
$f:D\to D$ is pointwise-recurrent if and only if $f$ is a
pointwise-periodic homeomorphism. Suppose now that the set $B(D)$ of
branchpoints of $D$ is discrete. Then it is proven in \cite{nag12}
that $f:D\to D$ is pointwise-recurrent if and only if all cutpoints
of $D$ (i.e., all points of $D\sm \End(D)$) are periodic.

The aim of this paper is to consider pointwise-recurrent maps on
generalized dendrites. We need a few definitions. Observe that in
Definitions~\ref{d:perset} and ~\ref{d:addma} we include no
topological requirements on either a set or a map.

\begin{dfn}[Periodic sets]\label{d:perset}
A set $A$ is said to be \emph{($f$-)periodic} if $A,$ $f(A),$
$\dots,$ $f^{n-1}(A)$ are pairwise disjoint while $f^n(A)\subset A$.
More generally, the union of $n$ pairwise disjoint sets $A_0, \dots,
A_{n-1}$ is said to be an \emph{($f$-)cycle of sets (of period $n$)}
if $f(A_i)\subset A_{i+1}, i=0, \dots, n-2$ and $f(A_{n-1})\subset
A_0$. Each set $A_i$ is then said to be a \emph{set from a cycle of
sets}.
\end{dfn}

Periodic singletons (orbits) are the simplest periodic sets (cycles
of sets).

\begin{dfn}[Adding machines]\label{d:addma}
Let $\Cc=\{C_0\supset C_1\supset \dots\}$ be a nested sequence of
$f$-cycles of sets of periods $m_n\nearrow \infty$ (clearly,
$m_{i+1}$ is a multiple of $m_i$ for any $i$). We say that $\Cc$
\emph{generates a weak adding machine
$C_\infty=\bigcap^\infty_{n=0}C_n$ (of type $(m_0, m_1, \dots)$}. If
the intersection of each \emph{nested} sequence of sets from the
cycles of sets $C_n$ is non-empty, then we call $C_\infty$ a
\emph{full weak adding machine}.
\end{dfn}

A weak adding machine is an $f$-invariant set. For a nested sequence
$\Cc=C_0\supset \dots$ of cycles of sets of periods $m_i$, choose a
nested sequence of sets from these cycles $\Rc=\{T^0\supset \dots\}$
and call it the \emph{root} of $C_\infty$; there are infinitely many
ways to choose the root. Once it chosen each set $X$ in each cycle
$C_n$ of sets from $\C$ acquires a natural index from $0$ to $n-1$
depending on the least power of $f$ mapping $T_n$ into $X$. We
denote sets from the cycle $C_i$ by setting $T^0_n=T^0$ and then
$T^i_n, 0\le i\le n-1$ so that $f^i(T^0_0)\subset T^i_n$. Clearly, a
sequence $\{T^{j_i}_i,$ $i=0, 1, \dots\}$ is nested if and only if
$j_{i+1}\equiv j_i \pmod {m_{i+1}}$ (if, for some $i$, $j_{i+1}\not
\equiv j_i \pmod {m_{i+1}}$, then $T^{j_{i+1}}_{i+1}\cap
T^{j_i}_i=\0$). Some \emph{nested} sequences $\{T^{j_i}_i, i=0, 1,
\dots\}$ may have empty intersections.

Set $Z_\infty=\Z_{m_0}\times \Z_{m_1}\times \dots$ and define
$H(m_0, m_1, \dots)=H\subset Z_\infty$ as the set of all sequences
$(j_0, j_1, \dots)\in Z_\infty$ with $j_{i+1}\equiv j_i \pmod
{m_{i+1}}$. Let $\tau:H\to H$ be such that $\tau(j_0, j_1,
\dots)=(j_0+1\pmod {m_0}, j_1+1\pmod {m_1}, \dots)$. The map $\tau$
models $f|_{C_\infty}$ for an adding machine $C_\infty$ of type
$(m_0, m_1, \dots)$ generated by $f$-periodic sets $C_0\supset
C_1\supset \dots$; to each non-empty intersection $\bigcap
T^{j_i}_i$ we associate the sequence $\ol{j}=(j_0, j_1, \dots)$. By
the above $\ol{j}\in H(m_0, \dots)$. This gives a map
$\psi:C_\infty\to H(m_0, \dots)$ of a weak adding machine to an
invariant subset $\psi(C_\infty)$ of $H(m_0, \dots)$. Clearly,
$\psi$ semiconjugates $f|_{C_\infty}$ with $\tau|_{\psi(C_\infty)}$.

\begin{dfn}[Models of adding machines]\label{d:perset1}
Suppose that cycles of sets $C_0\supset \dots$ generate a weak
adding machine $C_\infty$ of type $(m_0, \dots)$. Then $C_\infty$ is
said to be a \emph{topological weak adding machine} if $\psi$ is
continuous and one-to-one onto image, and \emph{topological adding
machine} if $\psi$ is a homeomorphism onto $H(m_0, \dots)$.
\end{dfn}

Lemma~\ref{l:addma} uses terminology and notation from
Definitions~\ref{d:perset} - \ref{d:perset1}. 

\begin{lem}\label{l:addma}
Suppose that $f:X\to X$ is a map of a topological space $X$ and that
cycles of sets $C_0\supset C_1\supset \dots$ generate a weak adding
machine $C_\infty$. Then:

\begin{enumerate}

\item if for each $i$ and for all $J>i$ the sets of $C_i$ are open in the relative topology of
of the corresponding sets of $C_i$ then the map $\psi$ is a
continuous map of $C_\infty$ onto $\psi(C_\infty)\subset H(m_0,
\dots)$;

\item if for each $j$ all sets in $C_j$ are compact then $\psi$ is a continuous map and
$\psi(C_\infty)=H(m_0, \dots)$.

\end{enumerate}

\end{lem}

\begin{proof}
(1) As the basis in $H(m_0, \dots)$ we can choose \emph{cylinders}
(sets consisting of sequences in $H(m_0, \dots)$ for which a few
initial parameters are fixed). Then for each such cylinder $K$ its
$\psi$-preimage is the appropriate set $B$ from $C_i$ intersected
with $C_\infty$. By the assumption there is an open set $U\subset X$
such that $U\cap C_i=B$. It follows that $U\cap C_\infty=B\cap
C_\infty=\psi^{-1}(K)$. Hence $\psi^{-1}(K)$ is open in $C_\infty$.

(2) Follows from (1), from the fact that nested sequences of compact
sets have non-empty intersections, and from the assumptions of the
lemma.
\end{proof}

Observe that $H(m_0, \dots)$ - and therefore any topological weak
adding machine associated with $H(m_0, \dots)$ - is uncountable.

\begin{dfn}\label{d:raycont}
A \emph{ray $R$} is the image of $\R_+\cup \{0\}$ under an embedding
$F$ into a topological space $X$. If $F(t)$ converges as $t\to
\infty$, we say that $R$ converges at infinity. If $X$ is a uniquely
arcwise connected locally arcwise connected topological space, then
we say that $X$ is \emph{ray complete} if every ray in $X$ converges
at infinity.
\end{dfn}

For a map $f$, set $F_n(f)=F_n=\{x:f^n(x)=x\}$ and
$D_n(f)=D_n=\bigcup^{n}_{i=1}F_i$.

\begin{thm}\label{t:mainexpandintr}
 Let $X$ be a uniquely arcwise connected locally arcwise connected
 topological space. Then a continuous map $f:X\to X$ is
 pointwise-recurrent if and only if all its cutpoints are periodic.
 Moreover, in this case the following holds.

 \begin{enumerate}

 \item The map $f$ is one-to-one; the set of all cutpoints of $X$ is fully invariant.

 \item The sets $F_n(f)$ and $D_n(f)$ are arcwise connected and
 closed for any $n$.

 \item An endpoint $x$ of $X$ is periodic or belongs to a topological
 weak adding machine (then $x$ is a limit point of a
 sequence of branchpoints of $X$).

 \item  If $X$ is ray complete (e.g., if $X$ is compact) then
 an endpoint of $X$ is periodic or belongs to a topological adding
 machine. If $X$ is a tree then there exists $N$ such that $f^N$ is
 the identity map.
 \end{enumerate}

\end{thm}

Let us describe a map $f$ from Theorem~\ref{t:mainexpandintr}. For
each $n$ with $D_n\subsetneqq D_{n+1}$, cycles of connected sets of
periods $n+1$ are added to $D_n$. Let $A$ be one of sets from such
cycle of sets $C$. Then $A$ is attached to $D_n$ at a point $x$ of
period $m\le n$ with $n+1=km$. There are $k>1$ sets from $C$
attached to $x$; they ``rotate'' around $x$ under iterations of
$f^m$ and have no points mapped to $D_n$ (in particular, $x$ is a
branchpoint of $X$ as there are at least two sets from $C$ and the
set $D_n$ which meet at $x$). As $n$ increases, the growth of $D_n$
can stop at some place to never resume; then the corresponding part
of $X$ consists of periodic points only (with bounded from above
periods). Otherwise the periods of sets like $A$ grow to infinity
which results in creation of recurrent points from topological weak
adding machines.

Theorem~\ref{t:mainexpandintr} does not hold for uniquely arcwise
connected spaces which are not locally arcwise connected. Indeed,
consider a compact topological space $X$ formed by a set of radii of
the unit circle whose arguments form a Cantor set $C\subset \uc$.
Define a minimal map $f$ on $C$ and then extend it onto $X$ so that
each radius $R_x$ defined by a point $x\in C$ maps to the radius
$R_{f(x)}$ defined by the point $f(x)$, and the map is an isometry
on $R_x$. Then all points of $X$ are recurrent.

\noindent\textbf{Acknowledgments.} The author would like to thank J.
Mayer, L. Oversteegen and L. Snoha for useful discussions of the
results of the paper. The author also would like to express his
gratitude to the referee whose report led to a significant
improvement of the paper and development of new tools allowing one
to better deal with maps continuous on arcs.


\section{Uniquely arcwise connected topological
spaces}\label{s:uniarcon}

From now on we \emph{\textbf{always}} consider a \uacs{} $X$. Call a
map $f:X\to X$ \emph{continuous on arcs} if, for any arc $I\subset
X$, the restriction $f|_I$ is continuous. From now on we
\emph{\textbf{always}} consider a map $f:X\to X$ continuous on arcs.
Continuity on arcs does not imply continuity.

\begin{example}\label{e:arconbad}
Set $X\subset \C$ to be the union of a closed interval $I$
connecting $(0, -1)$ and $(0, 0)$ and a countable collection of
closed intervals $J_k$ of lengths $\frac1k$ emanating from $(0, 0)$
and forming the angles $\frac{2\pi}k, k\ge 2$ with the positive
direction on $x$-axis. Clearly $X$ is a dendrite. Now, define a map
$f:X\to X$ as follows. First set $f(I)=(0, -1)$; in other words, we
assume that $I$ collapses to the point $(0, -1)$. To define $f$ on
each $J_k$, denote by $x_k$ the midpoint of $J_k$ for each $k$.
Denote by $A_k$ and $B_k$ the two closed subintervals into which
$x_k$ divides (except for the common endpoint $x_k$) the interval
$J_k$ so that $(0, 0)\in B_k$. Set $f|_{A_k}$ to be the identity map
and $f|_{B_k}$ to be a linear (with respect to the plane metric
restricted on $X$) map which stretches $B_k$ onto $I\cup B_k$). Then
not only is our map $f$ not continuous, but also even the set of all
$f$-fixed points is $(-1, 0)\cup (\bigcup A_k)$ which is not closed
while $f$ is clearly continuous on arcs. This shows the limitations
of conclusions which we can make by only assuming that $f$ is
continuous on arcs.
\end{example}

Another unpleasant property of maps continuous on arcs is that they
need not have continuous on arcs iterates. The author is grateful to
the referee for the following example.

\begin{example}\label{e:arconbad1}
Take the dendrite $X$ from Example~\ref{e:arconbad}. Take a sequence
$\{y_k=(0, -1+2^{-k}):k=0, 1, \dots\}$ and set $I_k=[y_k, y_{k+1}]$.
Construct now our map $f$ as follows. Let $f$ take the points
$\{y_k, k=1, 2, \dots\}$ to $(0, 0)$ while taking $y_0=(0, 0)$ to
$(0, -1)$. Map $I_0$ linearly onto $I$ and each $I_k$ onto $J_{k+1}$
so that the midpoint of $I_k$ maps to the endpoint of $J_{k+1}$ and
linearly otherwise. Thus, as $t$ moves down along $I_k$, the point
$f(t)$ moves along $J_{k+1}$ first from $(0, 0)$ out to the other
endpoint of $J_{k+1}$ and then back down to $(0, 0)$. Finally, map
$(0, -1)$ to $(0, 0)$, and leave $f$ as it was defined in
Example~\ref{e:arconbad} on $\bigcup^\infty_{k=1}J_{k+1}$.

Clearly, $f$ is continuous on arcs. However $f^2|_I$ is not
continuous. To see that, observe that for any $s\in I$ there is a
point $s_k\in I_k$ with $f^2(s_k)=s$. Notice also that $f|_I$ is
continuous and maps $I$ onto the entire $X$ (so that $X$ is a Peano
subset).
\end{example}

This shows that results on maps continuous on arcs require special
tools. As we see below, these tools are of one-dimensional nature.
They and based upon the fact that some other standard facts still
hold for maps continuous on arcs. E.g., let $A\subset X$ be arcwise
connected. Then $f(A)$ is also arcwise connected. Indeed, take two
points $f(x)\in f(A), f(y)\in f(A)$ and consider $f|_{[x, y]}$.
Since $f$ is continuous on arcs, the set $f([x, y])$, as a
continuous image of an arc, is arcwise connected as desired. Observe
also, that if $f$ is continuous on arcs then it is continuous on
\emph{trees} (finite unions on arcs in $X$).

\begin{dfn}\label{d:closonarc}
A set $Z$ is said to be \emph{closed on arcs} if for any arc $[a,
b]$ the intersection $[a, b]\cap Z$ is closed in $[a, b]$.
\end{dfn}

As an example of how this notion is used, let us prove
Lemma~\ref{l:clos1}.

\begin{lem}\label{l:clos1}
Suppose that $X$ is uniquely arcwise connected, $f:X\to X$ is
continuous on arcs, and $Y\subset X$ is an arcwise connected set
such that $f^k(Y)\subset Y$. If the set $F_n=\{x: x\in Y,
f^{kn}(x)=x\}$ is arcwise connected, then for any closed arc $I=[a,
b]\subset Y$ the intersection $F_n\cap I$ is a closed arc.
\end{lem}

In other words, $F_n$ is closed on arcs if it is considered as a
subset of $Y$.

\begin{proof}
Clearly, $F_n\cap I$ is an interval with endpoints, say, $u$ and
$v$, where $u$ ($v$) either belongs to $F_n\cap I$ or not. Observe
that $f|_I$ is continuous and one-to-one on $(u, v)$. Hence $f(I)$
is an arc with endpoints $f(u), f(v)$ such that $(f(u), f(v))\subset
P_f$. Moreover, $f|_{[f(u), f(v)]}$ is continuous and hence
$f^2|_{[u, v]}$ is continuous. Repeating this argument $n$ times, we
see that $f^n|_{[u, v]}$ is continuous and identity on $(u, v)$.
Hence $f^n(u)=u, f^n(v)=v$ and $u, v\in F_n$ as desired.
\end{proof}

Lemma~\ref{l:retract} allows us to ``project'' points in $X$ to its
subsets closed on arcs.

\begin{lem}\label{l:retract}
Let $Y\subset X$ be an arcwise connected set closed on arcs. Let
$z\notin Y$ be a point of $X$. Then there exists a unique point
$w\in Y$ such that $(w, z]\cap Y=\0$. Moreover, if $z$ and $z'$
belong to the same path component of $X\sm Y$, the corresponding
point $w$ serves both $z$ and $z'$.
\end{lem}

In the proof we repeatedly use the fact that $X$ is uniquely arcwise
connected.

\begin{proof}
Choose a point $x\in Y$ and consider $[x, z]$. Then for some point
$w$ we have that $[x, w]\subset Y$ while $(w, z]\cap Y=\0$. Let us
show that $w$ with these properties is unique. Suppose that $w'\in
Y, w'\ne w$ is such that $(w, z]\cap Y=\0$. Then for some point
$u\in [z, w]\cap [z, w']$ we must have that $(u, w]\cap (u, w']=\0$.
Connecting $w$ and $w'$ with an arc inside $Y$ we will get a
contradiction with the fact that $X$ is uniquely arcwise connected
as $[u, w]\cup [w, w']$ and $[w', u]$ are two distinct arcs
connecting $w'$ and $u$ and . Thus, $w$ is well-defined. The
remaining claim is left to the reader.
\end{proof}

Lemma~\ref{l:retract} leads to the following definition.

\begin{dfn}\label{d:project}
Denote the point $w$ from Lemma~\ref{l:retract} by $p_Y(z)$.
Moreover, for any point $z\in Y$ we set $p_Y(z)=z$.
\end{dfn}

In general the map $p_Y(z)$ is not continuous. E.g., take the
``Warsaw circle'' (see Example~\ref{e:warsaw}) and choose $Y=[a, b]$
to be a closed arc inside $I$. Then choose $z\in [(0,0), (0,1)]=K$.
It follows that if for points of $K$ the ``projection'' to $Y$ is,
say, $a$, then for all other close by points of $W_c$ the
``projection'' to $Y$ is $b$.

However it is easy to see that if the set $X'\subset X$ is locally
connected then $p_Y|_{X'}$ is continuous. Thus, if $Y\subset X'$
then $p_Y|_{X'}$  is a \emph{retraction}.

\begin{dfn}\label{d:connhall}
Let $E=\{e_1, \dots, e_k\}$ be a collection of points of $X$. Then
the smallest connected set $\ch(E)$ containing $E$ is called the
\emph{connected hall} of $E$.
\end{dfn}

Before we prove the next lemma observe that if $Y\subset X$ is a
tree then $f(Y)$ is a dendrite (i.e., a locally connected uniquely
arcwise connected compactum). In particular this implies that for
any dendrite $D\subset f(Y)$ all components (equivalently, path
components) of $f(Y)\sm D$ are open in $f(Y)$.

\begin{lem}\label{l:tree}
Let $E=\{e_1, \dots, e_n\}\subset X$ and set $Y=\ch(E)$. Let
$Z=\ch(f(E))$ and let $T=Y\cap f^{-1}(Z)$. Then $f(T)=Z$ and $f|_T$
can be extended over the entire $Y$ as a continuous map $F$ so that
on any component of
$Y\sm T$ the map $F$ is a constant. 
\end{lem}

Notice that $f|_T$ is continuous. Also, if $y\in Y$ is such that $F$
is not a constant on a neighborhood of $y$ then $y\in T$ and so in
fact $F(y)=f(y)$.

\begin{proof} The dendrite $f(Y)$ contains $Z$. Hence the map $p_Z$
on $f(Y)$ is a retraction. Define the map $F$ as $(p_Z\circ f)|_Y$.
Then $F$ coincides with $f$ on $T$. Moreover, continuity of $f$ on
$Y$, the fact that $f(Y)$ is a dendrite, and the above listed
properties of ``projections'' imply the rest of the lemma.
\end{proof}

The construction of the map $F$ from Lemma~\ref{l:tree} can be
iterated. This immediately yields Corollary~\ref{c:tree}.

\begin{cor}\label{c:tree}
Let $E=\{e_1, \dots, e_n\}\subset X$ and set $Y=\ch(E)$. Let
$Z_i=\ch(f^i(E)), i=0, 1, \dots$. Let $T_n\subset Y$ be a set of all
points $y\in Y$ such that $f^i(y)\in Z_i, i=0, 1, \dots, n$. Then
$f^n(T_n)=Z_n$ and $f^n|_{T_n}$ can be extended over the entire $Y$
as a continuous map $F_n$ so that on any component of
$Y\sm T_n$ the map $F_n$ is a constant. 
\end{cor}

This leads to Lemma~\ref{l:perio}.

\begin{lem}\label{l:perio}
Let $E=\{e_1, \dots, e_n\}\subset X$ and set $Y=\ch(E)$. Suppose
that $\ch(f^n(E))\supset Y$. Then there are periodic points of $f$
in $Y$ whose entire $f^n$-orbit is contained in $Y$.
\end{lem}

\begin{proof}
Consider a map $F_n:Y\to \ch(f^n(E))$ constructed in
Corollary~\ref{c:tree}. Then compose it with $p_Y$ to construct a
continuous map $g=p_Y\circ F_n:Y \to Y$. Take a fixed point $x$ of
$g$ (it is well known \cite{nad92} that such point exists). If $g$
is not a constant on a neighborhood of $x$ in $Y$ then it follows
from the construction that $f^n(x)=g(x)$ as desired. Otherwise
choose the open set $W$ of points attracted to $x$ (since $g$ is a
constant on a neighborhood of $x$, the set $W$ is open), and then
the component $U$ of $W$ containing $x$. It is well-known that the
(finite) boundary of $U$ maps to itself. This implies that there are
$g$-periodic points in $\bd(U)$. If one such point belongs to an
open set on which $g$ is a constant, then close by points of $U$
will not be attracted to $x$, a contradiction. Hence $g$ and $f^n$
coincide on all $g$-periodic points in $\bd(U)$ which completes the
proof.
\end{proof}

Lemma~\ref{l:perio} allows one to make conclusions about the
$f^n$-orbits of points under certain circumstances. To make such
conclusions we need the following definition.

\begin{dfn}\label{d:aset}
Given a map $g:X\to X$, and a point $y\in X$ with $g(y)\ne y$, let
\emph{$A_g(y)$ be the path component of $X\sm \{y\}$ containing
$g(y)$.} It follows that $z\in A_g(y)$ if and only if $y\notin [z,
g(y)]$.
\end{dfn}

In Corollary~\ref{c:lim-non-per0} we study maps without periodic arc
cutpoints.

\begin{cor}\label{c:lim-non-per0}
Let $f:X\to X$ be a map continuous on arcs without periodic arc
cutpoints and $x\in X$ be a point with $f^n(x)\ne x$. Then the
entire $f^n$-orbit of $x$ is contained in $A_{f^n}(x)\cup \{x\}$, so
that if $x$ is not periodic then $\orb_{f^n}(f^n(x))\subset
A_{f^n}(x)$.
\end{cor}

\begin{proof}
First observe that if $x$ is an arc endpoint then the claim holds
because then $A_{f^n}(x)=X\sm \{x\}$. Assume now that $x$ is not an
arc cutpoint. Suppose by way of contradiction that there exists the
minimal $m$ such that $f^{mn}(x)\notin A_{f^n}(x)$. 
By the assumption $f^{mn}(x)\ne x$. Set $E=\{x, f^n(x),
f^{n(m-1)}(x)\}$ and $Y=\ch(E)$. Then $\ch(f^n(E))\supset Y$. By
Lemma~\ref{l:perio} there is a periodic point $y\in Y$ with
$\orb_{f^n}(y)\subset Y$. Since $f^{mn}(x)\notin A_{f^n}(x)$ then
$y\notin E$ which implies that $y$ is an arc cutpoint, a
contradiction.
\end{proof}

In the interval case Corollary~\ref{c:lim-non-per0} deteriorates to
an obvious statement according to which if there are no interior
periodic points of $f:[0, 1]\to [0, 1]$ then all points of $(0, 1)$
map in the same direction under $f$.

\section{Proofs of main results}\label{s:limit-non-per}

We need the following definition inspired by that of a recurrent
point.

\begin{dfn}\label{d:incompo}
Consider a point $x$ of a \uacs{} $X$. Suppose that a  map $g:X\to
X$ is given. If for any $y\in X, y\ne x$ there exists $n>0$ such
that $x$ and $g^n(x)$ belong to the same path component of $X\sm
\{y\}$ then $x$ is said to \emph{return (to path components, under
$g$)} (or to be a \emph{returning (to path components, under $g$)}
point). If $x$ returns to components under any power of $g$ then we
say that $x$ \emph{totally returns (to path components, under $g$)}
(or is a \emph{totally returning (to path components, under $g$)}
point).
\end{dfn}

By a \emph{preperiodic} point we mean a non-periodic point which
eventually maps to a periodic point. We need the following simple
observation.

\begin{lem}\label{l:pownoprep}
If $g:X\to X$ is given and a point $x\in X$ is such that $x\ne
g(x)=g^2(x)$ then $x$ is not returning to path components under $g$.
If $y$ is totally returning then $y$ is not preperiodic.
\end{lem}

\begin{proof}
Choose $z$ separating $x$ from $g(x)$; it follows that $g(x)$ does
not return to the path component of $X\sm \{z\}$ containing $x$ and
proves the claim. Applying this claim to $y$ and $g^N$ with
sufficiently large $N$ completes the proof of the lemma.
\end{proof}

To prove lemmas leading to the proof of Theorem~\ref{t:mainexpand}
which implies Theorem~\ref{t:mainexpandintr} we make the following
Standing Assumption about the map we are working with.

\smallskip

\noindent\textbf{Standing Assumption}. \emph{We assume that $f:X\to
X$ is a continuous on arcs map such that all points totally return
to path components.}

\smallskip

Suppose that $Y\subset X$ is arcwise connected and such that
$f^N(Y)\subset Y$. Then $f^N|_Y$ is such that all points totally
return to path components. However as example~\ref{e:arconbad1}
shows we cannot guarantee that $f^N|_Y$ is continuous on arcs.
Still, Lemma~\ref{l:tree} and Corollary~\ref{c:tree} allow us to
work with powers of $f$.


The next key lemma is an important technical result.

\begin{lem}\label{l:imposs}
Suppose that $x'\in X$ is such that $f^n(x')=x'$. Then it is
impossible that for some $t\ne x'$ we have $[x', t]\subset [x',
f^n(t))$.
\end{lem}

\begin{proof}
Suppose otherwise. Then by Corollary~\ref{c:tree} there exists a
point $z=z_0\in (x', t)$ such that $f^n(z)=t$, a point $z_1\in (x,
z)$ such that $f^n(z_1)=z$, etc. The sequence $z_i$ is ordered on
$[x', t]$ in the sense of induced order so that $z_{i+1}$ separates
$z_i=f^n(z_{i+1})$ from $x'$ on $[x', t]$, and $z_i\to x$ for some
point $x\in [x', t]$. By Corollary~\ref{c:tree} $f^n([z_{i+1},
z_i])\supset [z_i, z_{i-1}]$. Thus for any $m$ we have $f^{nm}(x,
z_m)\supset (x, z)$.

Consider the union $Y=\bigcup_i f^{ni}(x, z]$. Since $(x, z]\subset
f^n(x, z]$, it follows that $f^n(Y)=Y$. Clearly, $Y\subset X$ is
uniquely arcwise connected. Let us show that $Y$ contains no
periodic points.  Indeed, suppose that $Y$ contains a periodic point
$u$. Then we can choose $N$ so big that $z_N$ is very close to $x$
and $(x, z_N]$ contains no points of $\orb_f(u)$. Since $f^{Nn}(x,
z_N]\supset (x, z]$, then the periodic point $u$ has eventual
preimages which do not belong to $\orb_f(u)$. As this contradicts
Lemma~\ref{l:pownoprep}, we see that indeed $Y$ contains no periodic
points. By Corollary~\ref{c:tree} this implies that, e.g., $z_1$
does not totally return to path components, a contradiction.
\end{proof}

We will need the following simple fact.

\begin{lem}\label{l:fixptree}
A continuous map of a tree to itself has a fixed point. In particular,
suppose that $Z\subset X$ is a tree with all its cutpoints periodic such
that $f^n$ maps its endpoints map to $Z$. Then there is an $f^n$-fixed point
in $Z$.
\end{lem}

\begin{proof}
The first claim of the lemma is well-known (see, e.g., \cite{nad92}).
To prove the second observe that $f$ is one-to-one on its cutpoints. Since
by the assumption $f$ is continuous on $Z$, it follows that in fact $f$ maps
$Z$ homeomorphically onto its image $f(Z)$. Hence, the $f$-images of endpoints
of $Z$ are the endpoints of $f(Z)$. Repeating this argument $n$ times we see that
$f^n|Z$ is a homeomorphism of $Z$ onto $f^n(Z)$ and that the $f^n$-images of the
endpoints of $Z$ are the endpoints of $f^n(Z)$. By the assumption this implies that
$f^n(Z)\subset Z$. Hence by the first claim of the lemma there are $f^n$-fixed points
in $Z$.
\end{proof}

Though assumptions on continuity of $f$ are weak, we prove for $f$
some standard properties; recall, that $F_n(f)$ is the set of all
$f^n$-fixed points of $f$. Thus, if $Y\subset X$ is such that
$f^n(Y)\subset Y$, then $F_k(f^n|_Y)$ is the set of all
$f^{nk}$-fixed points in $Y$.

\begin{lem}\label{l:fixclos}
Let $Y\subset X$ be arcwise connected and such that $f^k(Y)\subset
Y$. Then the set $F_n(f^k|_Y)$ is arcwise connected and closed on
arcs. The set $F_1(f^k|_Y)$ is non-empty (and so all sets
$F_n(f^k|_Y)$ are non-empty).
\end{lem}

If $z\notin F_n(f^k|_Y)$ then by Lemma~\ref{l:fixclos} we can define
the point $p_{F_n(f^k|_Y)}(z)=x_n(z)$ for which the path component
of $Y\sm \{x_n(z)\}$ which contains $z$ contains no points of
$F_n(f^k|_Y)$ (in particular, $[z, x_n(z))\cap F_n(f^k|_Y)=\0$).

\begin{proof}
For brevity throughout the proof we set $F_i(f^k|_Y)=F_i, i=1, 2,
\dots$. Let us assume that $F_n\ne \0$. First we show that $F_n$ is
arcwise connected. Indeed, otherwise there are two points $x, y\in
F_n$ such that $[x, y]\not\subset F_n$. Choose a point $z\in (x, y)$
such that $f^{kn}(z)\ne z$. Clearly, then at least one these two
statements holds: (1) $[x, z]\subset [x, f^{kn}(z))$, or (2) $[y,
z]\subset [y, f^{kn}(z))$. By Lemma~\ref{l:imposs} this contradicts
our Standing Assumption. Thus, $F_n$ is arcwise connected. By
Lemma~\ref{l:clos1} this implies that $F_n$ is closed on arcs.

We claim that $F_n\ne \0$ for some $n$. Assume otherwise and
consider $x\in Y$. Then $f^k(x)\ne
x$, and by Corollary~\ref{c:lim-non-per0} $\orb_{f^k}(x)\subset A_{f^k}(x)$.
If for a point $y\notin A_{f^k}(x)$ there exists an integer $n$ with
$f^{nk}(y)=x$, then $\orb_{f^k}(x)\subset A_{f^k}(x)$ implies that
$y$ does not return to path components
under $f^k$, a contradiction. Hence $x\notin f^{kn}(Y\sm A_{f^k}(x))$ for
every $n$. We claim that then $f^{kn}(Y\sm A_{f^k}(x))\subset A_{f^k}(x)$ for every $n$.
Indeed, otherwise we can choose a point $y\in Y\sm A_{f^k}(x)$ with
$f^{kn}(y)\notin A_{f^k}(x)\cup \{x\}$. Since $f^{kn}(x)\in A_{f^k}(x)$ then by
Lemma~\ref{l:tree} there exists a point $z\in [y, x]$ with
$f^{kn}(z)=x$, a contradiction. Hence $F_n\ne \0$ for some $n$.

Now, take $x\in F_n$ and consider the set $\orb_{f^k}(x)\subset F_n$. Then
$Z=\ch(\orb_{f^k}(x))\subset F_n$ is a tree. By Lemma~\ref{l:fixptree}
there are $f^k$-fixed points in $Z$. Hence $F_1\ne \0$.
\end{proof}

Recall that $D_n(f)=D_n$ is the union of set $F_n(f)=F_n, n=1,
\dots, n$.

\begin{thm}\label{t:mainexpand}
 If $X$ is uniquely arcwise connected and $f:X\to X$ is con\-ti\-nuous on arcs
 then all points of $X$ totally return to path components under $f$
 if and only if all arc cutpoints of $X$ are periodic. Moreover, then the following holds.

 \begin{enumerate}

 \item The map $f$ is one-to-one; the set of all arc cutpoints of $X$ is fully invariant.

 \item The sets $F_n(f)$ and $D_n(f)$ are arcwise connected for any $n$.

 \item An endpoint $x$ of $X$ is periodic or belongs to a weak adding
 machine generated by cycles of arcwise connected sets (then $x$ is a
 limit point of a sequence of branchpoints of $X$).

 \item If $X$ is ray complete (e.g., if $X$ is compact) then
 an endpoint of $X$ is periodic or belongs to a full weak adding
 machine. If $X$ is a tree then there exists $N$ such that $f^N$ is
 the identity map.
 \end{enumerate}

\end{thm}

\begin{proof} Denote by $\ar$ the set of all arc cutpoints of $X$. Also,
let $P_f=\bigcup D_n$ be the set of all periodic points of $f$.
First we prove that \emph{if all points of $X$ totally return to
path components under $f$ then $\ar\subset P_f$}. By
Lemma~\ref{l:fixclos} $F_1\ne \0$ and since $F_1\subset F_i$, then
$D_n$ is arcwise connected for any $n$.  Hence $P_f$ is invariant
and uniquely arcwise connected. Let us show that $\ar\subset P_f$.

Indeed, otherwise there exists an arc cutpoint $x\notin P_f$ and a
non-degenerate path component $A$ of $X\sm \{x\}$ disjoint from
$P_f$. Take a point $y\in A$. Since $y$ returns to path components
under $f$, there exists $N$ such that $f^N(y)\in A$. Connect $x$ and
a fixed point $a\in P_f$ to create an interval $[a, x]$ which
intersects $P_f$ over an interval $[a, b]$ or over an interval $[a,
b)$. Denote by $B$ the path component of $X\sm \{b\}$ containing
$x$.

We claim that $f^N(b)=b$. Indeed, $[a, b)\subset P_f$, hence
$f|_{[a, b)}$ is one-to-one which implies that in fact $f|_{[a, b]}$
is one-to-one. Repeatedly applying this, we see that $f^N[a, b]=[a,
f^N(b)]$ and that $[a, f^N(b))\subset P_f$. If $f^N(b)\in B$ then
there are points of $P_f$ close to $f^N(b)$ which belong to $B$, a
contradiction. If $f^N(b)\notin B$ then, since $f^N(y)\in A$, we
have $[f^N(b), f^N(y)]\supset [f^N(b), x]\supset [f^N(b), b]$. Since
$X$ is uniquely arcwise connected, an interval $(z, b)$ of points of
$[a, b)$ is contained in $(f^N(b), b)$ and hence in $f^N[b, y]$.
Since by Lemma~\ref{l:pownoprep} the map $f$ has no preperiodic
points, we arrive at a contradiction.

This implies that no point of $B$ ever maps to $b$. We claim that
$f^N(B)\subset B$. Indeed, $f^N(y)\in A\subset B$. If now there is a
point $d\in B$ with $f^N(d)\notin B$ then by Lemma~\ref{l:tree}
there is a point $u\in B$ with $f^N(u)=b$ again contradicting
Lemma~\ref{l:pownoprep}. Thus, $f^N(B)\subset B$. Since $A$ contains
no periodic points of $f^N$, by Lemma~\ref{l:fixclos} it contains
some points which do not totally return to path components under
$f^N|_A$, and hence do not totally return to path components under
$f$, a contradiction. This completes the proof of the fact that if
all points of $X$ return to path components under $f$ then $P_f$
contains all arc cutpoints of $X$.

We denote by $\ar$ the set of all arc cutpoints of $X$. Assume now
that all points $x\in \ar$ are periodic. Then $f$ is one-to-one on
$\ar$. Let $x\ne y$ but $f(x)=f(y)=z$. If no point $t\in (x, y)$
maps to $z$ we can choose $t\in (x, y)$ and observe that points $z$
and $f(t)]$ can be connected with two arcs, $f[x, t]$ and $f[t, y]$.
If there exists $t\in (x, y)$ with $f(t)=z$ we can apply the same
argument to $[t, y]$. Thus, $f$ is one-to-one, and hence all powers
of $f$ are one-to-one (in particular,  for any closed arc $I=[a,
b]\subset X$ and any $N$, the map $f^N|_I$ is a homeomorphism onto
image).

This immediately implies that $\ar$ is forward invariant. On the
other hand if an arc endpoint $x$ maps to an arc cutpoint $f(x)$
then $f(x)$ is periodic of period, say, $n$, and $f(x)$ has two
distinct preimages: $x$ and $f^{n-1}(f(x))$ (by the above
$f^{n-1}(f(x))$ is an arc cutpoint of $X$ and hence
$f^{n-1}(f(x))\ne x$), a contradiction. Hence $\ar$ is fully
invariant (both its image and its preimage are contained in it).

We claim that the set $F_N$ is arcwise connected for any $N$.
Indeed, if $x, y\in F_N$, then $F^N|_{[x, y]}$ is a homeomorphism
onto $[x, y]$ with all points being periodic. By
Lemma~\ref{l:interval} this implies that $f^N|_{[x, y]}$ is the
identity map and hence $[x, y]\subset F_n$. Thus, $F_N$ is arcwise
connected. By Lemma~\ref{l:clos1} this implies that $F_N$ is closed
on arcs. Moreover, $P_f\ne \0$ implies by Lemma~\ref{l:fixptree}
that $F_1\ne \0$. Then $D_n$ is arcwise connected for any $n$. Since
each $F_i$ is closed on arcs, then so is $D_n$ for any $n$.

Let us show that all points of $X$ totaly return under $f$. We may
assume that $x$ is an arc endpoint of $X$. Suppose that a number $n$
and a point $y\ne x$ are given. Choose points $u, v\in (x, y)$ so
that $x<u<v<y$ in the induced order on $[x, y]$. Choose a number $k$
such that $f^k(u)=u, f^k(v)=v$. Then $f^{ik}$ is identity on $[u,
v]$. Since$f^{ik}$ is continuous on arcs and one-to-one then
$f^{ik}(x)$ belongs to the arc component of $X\sm \{y\}$ containing
$x$. Hence, $x$ totally returns to path components under $f$ as
desired.

Let us prove claims (1)-(4) assuming that $\ar\subset P_f$ (and
hence, by the above, all points of $X$ totally return to path
components under $f$). We have already proven (1) and (2) for such
maps. To prove (3) we first make some observations. Take a path
component $A_n$ of $X\sm D_n$. Choose the smallest $N$ with
$f^N(A_n)\cap A_n\ne \0$. Choose $z'\in A_n$ and let
$p_{D_n}(z')=t^0_n$. Then $A_n$ is a path component of $X\sm
\{t^0_n\}$ disjoint from $D_n$. Since $f$ has no preperiodic points,
no point of $A_n$ ever maps to $D_n$, and so $f^N(A_n)\subset A_n$.
Clearly, $f^i(A_n)$ is a subset of a path component $T^i_n$ of $X\sm
D_n$ and by the choice of $N$ the sets $T^i_n, i=0, \dots, N-1$ are
all pairwise disjoint. Thus, the sets $A_n=T^0_n, T^i_n, \dots,
T^{N-1}_n$ form a cycle of sets. Similar to the above, for each set
$T^i_n$ there is a unique point $t^i_n=p_{D_n}(T^i_n)\in D_n$ such
that $T^i_n$ is a path component of $X\sm \{t^i_n\}$.

Let the period of $t^0_n$ be $m$. To prove that $t^0_n$ is an arc
branchpoint, choose a point $u\in T^0_n$ and set $I=[t^0_n, u]$.
Since $u$ is an arc cutpoint, $u$ is periodic of period $km$ with
$k>1$. It follows that $f^{km}|_I$ is the identity map and the set
$Y=\bigcup^{k-1}_{i=0} f^{im}(I)$ is a finite tree on which $f^{km}$
is the identity map. If $t^0_n$ is not an arc branchpoint of $Y$,
then the fact that $f^m(t^0_n)=t^0_n$ implies that a small subarc of
$Y$ with an endpoint $t^0_n$ consists of $f^m$-fixed points, a
contradiction with the fact that all $f^m$-fixed points are
contained in $D_n$. Thus, $t^0_n$ is an arc branchpoint of $X$.
Similarly, all points $t^i_n$ are branchpoints of $X$.

Now, let $z$ be an arc endpoint of $X$ which is not periodic. By the
above for any $n$ we can choose an arc component $T^0_n$ of $X\sm
D_n$ so that $z\in T^0_n$ and then the cycle of the sets
$C_n=T^0_n\cup \dots \cup T^{N-1}_n$ for some $N$. As the number $n$
grows, we will find a nested sequence of cycles of sets $C_0\supset
C_1\supset \dots$ containing $z$ of periods $m_0<m_1<\dots$. To show
that $C_\infty=\bigcap C_i$ is a weak adding machine, it suffices to
show that a nested sequence of sets $T^{j_0}_0\supset
T^{j_1}_1\supset \dots$ from cycles of sets $C_i$ is such that the
intersection $Z=\bigcap T^{j_i}_i$ is either empty or a singleton.
Indeed, otherwise $Z$ is a non-degenerate arcwise connected subset
of $X$ which is \emph{wandering} (i.e., all its images are pairwise
disjoint). Clearly, there are arc cutpoints of $X$ in $Z$. This
contradicts the periodicity of all arc cutpoints of $X$ and
completes the proof of (3).

To prove (4), take a nested sequence of sets $T^{j_0}_0\supset
T^{j_1}_1\supset \dots$ and the points $t^{j_i}_i$ defined above.
Then there is a unique ray $R$ connecting the points $t^{j_i}_i$. If
$X$ is ray complete then the intersection $Z=\bigcap T^{j_i}_i$ is
non-empty because it contains the point to which $R$ converges at
infinity. Finally, the claim in (4) about trees immediately follows
from (3) because trees have finitely many branchpoints.
\end{proof}

Now let us prove that Theorem~\ref{t:mainexpand} implies
Theorem~\ref{t:mainexpandintr}. Lemma~\ref{l:powrec} shows how
recurrent and totally returning points are related.

\begin{lem}\label{l:powrec}
If $f:X\to X$ is a continuous map of $X\in \U$ and $x$ is a
recurrent point of $f$ then $x$ is totally returning.
\end{lem}

\begin{proof}
Choose $y\ne x$ and denote by $A$ the component of $X\sm \{y\}$
containing $x$. Choose a small neighborhood $B$ of $x$ so that
$B\subset A$. Finally, suppose that a positive integer $N$ is given.
Since $y$ is recurrent, then by Theorem~\ref{t:recpow} there exists
$n$ such that $f^{Nn}(x)\in B\subset A$ as desired. 
\end{proof}

\noindent\emph{Proof of Theorem~\ref{t:mainexpandintr}}.\, First
observe that continuous maps are continuous on arcs. This and
Lemma~\ref{l:powrec} imply that Theorem~\ref{t:mainexpand} holds in
our setting. By Lemma~\ref{l:cutarcut},
Theorem~\ref{t:mainexpand}(1) implies
Theorem~\ref{t:mainexpandintr}(1). Clearly,
Theorem~\ref{t:mainexpand}(2) and continuity of $f$ imply
Theorem~\ref{t:mainexpandintr}(2). To prove
Theorem~\ref{t:mainexpandintr}(3) we need to show that the weak
adding machine $C_\infty$ from Theorem~\ref{t:mainexpand}(3) is
topological. Suppose that $C_\infty$ is generated by cycles of sets
$C_i, i=0, 1, \dots$ of periods $N_i, i=0, 1, \dots$. By
Lemma~\ref{l:addma} it suffices to show that sets from cycles of
sets $C_i$ are open in $C_i$ in relative topology for any $i$.
However this follows from Lemma~\ref{l:cutarcut}. Finally,
Theorem~\ref{t:mainexpandintr}(4) immediately follows from
Theorem~\ref{t:mainexpand}(4). \hfill\qed


In conclusion observe that a generalized dendrite $X$ admits a
canonical \emph{ray completion} $\widehat X$. A sketch of the
construction follows. Consider all rays in $X$ which do not converge
at infinity. Two such rays $R_1, R_2$ have either coinciding (from
some moment on), or disjoint (from some moment on) tails. In the
former case we consider them equivalent. To each class of
equivalence we associate a point of $\widehat X$ called a
\emph{point at infinity}. Define $\widehat X$ as the union of $X$
and the just defined points at infinity; as neighborhoods of those
points we take components $C$ of sets $X\sm \{b\}$ where $b$ is a
point of $X$ united with all points at infinity defined by rays
contained in $C$. It is easy to see that the space $\widehat X$ is a
ray complete generalized dendrite.

A pointwise-recurrent continuous map $f:X\to X$ can be extended to a
point\-wise-recurrent continuous map $\hat f:\widehat X\to \widehat
X$ of the ray completion $\widehat X$ of $X$. Then $f:X\to X$ can be
viewed as a result of removing from $\widehat X$ of a few backward
orbits of endpoints of $X$. It is not necessarily so that removed
points belong to topological adding machines; some removed points my
be periodic. Removing a periodic endpoint creates a ray in $X$ which
does not converge at infinity and is such that its tail consists of
points of the same period. The space $\widehat X$ may be a dendrite
or even a tree.

\bibliographystyle{amsalpha}

\begin{thebibliography}{1}

\bibitem[AEO07]{aeo07} G. Acosta, P. Eslami, L. Oversteegen, {\it On open maps between dendrites},
Houston Journal of Mathematics \textbf{33} (2007), 753--770.

\bibitem[ALM00]{alm00} L. Alseda, J. Llibre, M. Misiurewicz, {\it Combinatorial Dynamics and Entropy in
Dimension One}, Adv. Ser. in Nonlinear Dynamics \textbf{5}, World
Scientific, Singapore (2000).

\bibitem[Bes88]{bes88} M. Bestvina, \emph{Degenerations of the hyperbolic space}, Duke
Math. J. \textbf{56} (1988), no. 1, 143-–161.



\bibitem[BFMOT11]{bfmot11} A. Blokh, R. Fokkink, J. Mayer, L. Oversteegen, E. Tymchatyn
\emph{Fixed point theorems for plane continua with applications},
Memoirs of the AMS \textbf{224} (2013), no. 1053.

\bibitem[Blo80s]{blo80xx} A. Blokh, {\it On Dynamical Systems on One-Dimensional
Branched Manifolds. 1, 2, 3} (in Russian), Theory of Functions,
Functional Analysis and Applications, Kharkov, {\bf 46} (1986),
8--18;  {\bf 47} (1986), 67--77; {\bf 48} (1987), 32--46.








\bibitem[ES45]{es45} R. Erd\"os, A. H. Stone, \emph{Some remarks
onalmost periodic transormations}, Bull. Amer. Math. Soc.
\textbf{51}(1945), 126--130.

\bibitem[Got44]{got44} W. H. Gottschalk, \emph{Powers of
homeomorphisms with almost periodic properties}, Bull. Amer. Math.
Soc. \textbf{50}(1944), 222--227.

\bibitem[HY88]{hy88} J. Hocking, G. Joung, \emph{Topology (Dover Books on
mathematics)}, Dover (1988).

\bibitem[KP98]{kp98} B. Kolev, M.-C. P\'erou\`eme, \emph{Recurrent surface homeomorphisms},
Math. Proc. Cambridge Philos. Soc. \textbf{124} (1998), 161--168.

\bibitem[MO90]{mo90} J. Mayer, L. Oversteegen, \emph{A topological
characterization of ${\bf R}$-trees}, Trans. Amer. Math. Soc.
\textbf{320} (1990), no. 1, 395–-415.

\bibitem[MNO92]{mno92} J. Mayer, J. Nikiel, L. Oversteegen,
\emph{Universal spaces for $\R$-trees}, Trans. Amer. Math. Soc.
\textbf{334} (1992), 411--432.

\bibitem[Mon37]{mon37} D. Montgomery, \emph{Pointwise periodic homeomorphisms},
Amer. J. Math. \textbf{59} (1937), 118--120.

\bibitem[MS84]{ms84} J. Morgan, P. B. Shalen, \emph{Valuations,
trees, and degenerations of hyperbolic structures: I} , Ann. of
Math. (2) \textbf{122} (1984), 401--476.

\bibitem[MS88]{ms88} J. Morgan, P. B. Shalen, \emph{Degenerations of
hyperbolic structures. II. Measured laminations in $3$-manifolds},
Ann. of Math. (2) \textbf{127} (1988), no. 2, 403–-456.

\bibitem[MT89]{mt89} P. Minc, W. Transue, {\it Sharkovskii's theorem for
hereditarily decomposable chainable continua}, Trans. Amer. Math.
Soc. \textbf{315}(1989), 173--188.

\bibitem[Nad92]{nad92} S. Nadler, {\it Continuum theory: An
Introduction}, Chapman and Hall (1992)

\bibitem[Nag12]{nag12} I. Naghmouchi, {\it Pointwise-recurrent
dendrite maps}, Erg. Th. and Dyn. Syst. \textbf{33} (2013),
1115--1123.

\bibitem[Nik89]{nik89} J. Nikiel, \emph{Topology on pseudo-trees and
applications}, Mem. Amer. Math. Soc. \textbf{416} (1989).



\bibitem[OT90]{ot90} L. G. Oversteegen, E. D. Tymchatyn, \emph{Recurrent homeomorphisms on
$\R^2$ are periodic}, Proc. Amer. Math. Soc. \textbf{110} (1990),
1083--1088.

\bibitem[Sha64]{sha64} A. N. Sharkovsky, {\it Co-existence of the cycles of
a continuous mapping of the line into itself}, Ukrain. Mat. Zh.
\textbf{16} (1964), 61--71.

\bibitem[Sha64a]{sha64a} A. N. Sharkovsky, {\it Non-wandering points and
the center of a continuous map of the line into itself} (in
Ukrainian), Dop. Acad. Nauk Ukr. RSR, Ser. A (1964), 865--868.

\bibitem[Sha66]{sha66} A.N. Sharkovsky, {\em The behavior of a map in a
neighborhood of an attracting set }(in Russian), Ukr. Math. J., {\bf
18} (1966), 60--83.

\bibitem[Sha66a]{sha66a} A.N. Sharkovsky, {\em The partially ordered system of
attracting sets,} Soviet Math. Dokl., {\bf 7} (1966), 1384--1386.

\bibitem[Sha67]{sha67} A. N. Sharkovsky, {\it On a theorem of G. D. Birkhoff},
(in Russian), Dop. Acad. Nauk Ukr. RSR, Ser. A (1967), 429--432.

\bibitem[Sha68]{sha68} A.N. Sharkovsky, {\em Attracting sets containing no cycles
}(in Russian), Ukr. Math. J., {\bf 20} (1968), 136--142.

\bibitem[Thu88]{thu88} W.~Thurston, \emph{On the geometry and dynamics of
diffeomorphisms of surfaces}, Bull. Amer. Math. Soc. (N.S.)
\textbf{19} (1988), no. 2, 417–-431.

\bibitem[Wea72]{wea72} N. Weaver, \emph{Pointwise periodic homeomorphisms of continua}, Ann.
Math. \textbf{95} (1972), 83--85.


\end{thebibliography}

\end{document}